\documentclass[11pt,letterpaper]{amsart}
\usepackage{xyz}

\title{Asymptotic-Type Dimension Bounds\\ through Combinatorial Approaches}

\author{Jing Yu \textsuperscript{\dag}}\thanks{\textsuperscript{\dag} Shanghai Center for Mathematical Sciences, Fudan University, Shanghai, China. Email: {jyu@fudan.edu.cn}.}
\author{Xingyu Zhu \textsuperscript{\ddag}}\thanks{ \textsuperscript{\ddag} Department of Mathematics, Michigan State University, East Lansing, MI, USA. Email: {zhuxing3@msu.edu}.}

\begin{document}
\begin{abstract}
We develop a probabilistic framework for large-scale dimension bounds in metric geometry, based on padded decompositions, randomized ball carving on net graphs, and the Lov\'asz Local Lemma. For metric measure spaces with volume doubling constant $\cd$, we prove the sharp bound
$\asandim(X)\le \andim(X)\le \floor{\log_2 \cd}$.
In particular, if $(M,g)$ is a complete Riemannian $n$-manifold with $\Ric_g\ge 0$, then $\asdim(M)\le n$, thereby settling a question of Papasoglu on manifolds with nonnegative Ricci curvature. We also show that if $(X,\dist,\meas)$ is proper, volume noncollapsed, and has polynomial volume growth rate $\rho^V(X)$, then
$\asdim(X)\le \floor{\rho^V(X)}$.
Moreover, the corresponding control function can be chosen to have polynomial growth. This extends Papasoglu's sharp asymptotic-dimension bound from graphs of polynomial growth to a metric-measure setting. As applications, we study equality in the polynomial-growth bound for universal covers of nilmanifolds, and under nonnegative Ricci curvature we relate the equality case in the volume-doubling bound to Gromov largeness, obtaining in particular a consequence for complete manifolds with positive scalar curvature.
\end{abstract}

\maketitle
\section{Introduction}

Large-scale dimension invariants such as the asymptotic dimension, the asymptotic Assouad--Nagata dimension, and the Assouad--Nagata dimension play a central role in coarse and metric geometry. Finite asymptotic dimension has important applications to the Novikov conjecture and, in turn, to positive scalar curvature \cites{YuNovikov,Asphericalasdim}, while finite Assouad--Nagata dimension implies strong Lipschitz extension properties \cite{LangNdim}. Quantitative upper bounds for these dimensions are therefore of independent interest; for instance, they bound the topological dimension of asymptotic cones \cite{asconeandim}. See also \cite{GuyAssouadtype} for further applications. We write $\asdim$, $\asandim$, and $\andim$ for these notions; see Section~\ref{sec:pre}.

The purpose of this paper is to develop a probabilistic proof framework for sharp and growth-controlled large-scale dimension bounds. Our approach is based on padded decompositions, randomized ball carving on net graphs, and the Lov\'asz Local Lemma. 
It is inspired by recent work of Bernshteyn and the first-named author \cite{BY23} on graphs, but is adapted here to metric and metric measure spaces.

A triple $(X,\dist,\meas)$ is called a \emphd{metric measure space} if $(X,\dist)$ is a complete, separable metric space and $\meas$ is a nontrivial Borel measure on $X$ with full support. By ``nontrivial'' we mean that there exists a ball $B\subset X$ such that $0<\meas(B)<\infty$. We say that $\meas$ is \emphd{volume doubling} if there exists a constant $\cd\ge 1$ such that for every $x\in X$ and every $r>0$,
\[
\meas(B_{2r}(x))\le \cd\,\meas(B_r(x)).
\]
The constant $\cd$ is called the \emphd{volume doubling constant}. Our first main result is the following sharp bound.

\begin{theorem}\label{thm:VD}
Let $(X,\dist,\meas)$ be a metric measure space. If $\meas$ is volume doubling with doubling constant $\cd$, then
\[
\andim(X)\le \floor{\log_2 \cd}.
\]
\end{theorem}

For doubling metric spaces, Le Donne and Rajala proved the sharp estimate
\[
\andim(X)\le \floor{\log_2 N},
\]
where $N$ is the metric doubling constant of $X$ \cite{RajalaLeDonne}*{Theorem~1.1}, i.e. $N$ is the smallest positive integer such that for every $r>0$, any ball of radius $2r$ can be covered by at most $N$ balls of radius $r$. Even though volume doubling implies metric doubling and vice versa \cites{DoublingCpt,DoublingNoncpt},  Theorem~\ref{thm:VD} is not a formal corollary of this result, since the metric doubling constant obtained from a volume doubling measure need not admit a sharp bound in terms of $\cd$. Our proof works directly with the volume doubling constant.

Theorem~\ref{thm:VD} has an immediate geometric consequence, which settles a question raised by Papasoglu:  

\begin{question}[{\cite[Question~4.3]{papasoglu2021polynomial}}]
Let $M^n$ be a complete Riemannian manifold of nonnegative Ricci curvature. Is it true that $\asdim(M^n) \le n$? 
\end{question}

Indeed, if $(M,g)$ is a complete Riemannian $n$-manifold with $\Ric_g\ge 0$, then Bishop--Gromov implies that $\vol_g$ is volume doubling with doubling constant at most $2^n$, and hence
\[
\asdim(M)\le \asandim(M)\le \andim(M)\le n.
\]
Thus Papasoglu's question has an affirmative answer in full generality.

Our second main result concerns spaces of polynomial volume growth. Let $\rho^V(X)$ denote the polynomial volume growth rate of $(X,\dist,\meas)$; see Definition~\ref{def:VolG}. We say that $(X,\dist,\meas)$ is \emphd{volume noncollapsed} if
\[
v\defeq \inf_{x\in X}\meas(B_{1/2}(x))>0.
\]

\begin{theorem}\label{thm:measure}
Let $(X,\dist,\meas)$ be a metric measure space. If $X$ is proper, has polynomial volume growth, and is volume noncollapsed, then
\[
\asdim(X)\le \floor{\rho^V(X)}.
\]
Moreover, the corresponding control function for $\asdim(X)$ can be chosen to have polynomial growth.
\end{theorem}

Thus our two main theorems interact with Papasoglu's work in two distinct ways: Theorem~\ref{thm:VD} settles his Question~4.3 on complete manifolds with nonnegative Ricci curvature, while Theorem~\ref{thm:measure} extends his sharp asymptotic-dimension bound from graphs of polynomial growth to a metric-measure setting \cite[Corollary~3.3]{papasoglu2021polynomial}. The conclusion cannot in general be strengthened by replacing $\asdim$ with $\asandim$ or $\andim$; see \cite[Theorem~3.5 and Example~3.4]{papasoglu2021polynomial}.

We also study equality cases and geometric consequences. For universal covers of nilmanifolds, equality in Theorem~\ref{thm:measure} forces the underlying nilmanifold to be diffeomorphic to a torus. In a different direction, if $(M,g)$ is a complete Riemannian $n$-manifold with $\Ric_g\ge 0$ and
\[
\inf_{p\in M}\vol_g(B_1(p))>0,
\]
then equality in Theorem~\ref{thm:VD} is equivalent to Gromov largeness in terms of volume.

\begin{proposition}\label{prop:Euc}
Let $(M,g)$ be a Riemannian $n$-manifold with $\Ric_g\ge 0$ that is volume noncollapsed, i.e.\ $\inf_{p\in M}\vol_g(B_1(p))>0$. Then $\asandim(M)=n$ (or $\asdim(M)=n$) if and only if $M$ is large in the sense of Gromov, namely
\[
\sup_{x\in M}\vol_g(B_r(x))=\omega_n r^n
\qquad\text{for every } r\ge 0.
\]
\end{proposition}

As an immediate consequence, we obtain the following scalar-curvature corollary.

\begin{corollary}\label{cor:psc}
If $(M,g)$ is an $n$-dimensional complete noncompact manifold with $\Ric_g\ge 0$, $\Sc_g\ge 2$, and
\[
v\defeq \inf_{x\in M}\vol_g(B_1(x))>0,
\]
then
\[
\asandim(M)\le n-1.
\]
\end{corollary}

The common ingredient in the proofs of Theorems~\ref{thm:VD} and \ref{thm:measure} is an equivalent formulation of large-scale dimension in terms of padded decompositions. Starting from a net $T$ in $X$, we consider a net graph $G^M(X,T)$ and run a randomized ball-carving procedure on this graph. A suitable choice of random radii, together with the Lov\'asz Local Lemma, produces padded decompositions with controlled multiplicity and diameter. In the doubling setting this yields sharp linear control, while in the polynomial-growth setting it yields polynomial control. A complementary simplicial argument in the spirit of Gromov \cite{VolBddCohom} is discussed in Appendix~\ref{sec:Gromov}; it also interacts naturally with the scalar-curvature discussion.

The paper is organized as follows. In Section~\ref{sec:pre} we review the relevant notions from metric geometry, introduce padded decompositions, and record the form of the Lov\'asz Local Lemma that we use. Section~\ref{sec:Nagata} proves the sharp Assouad--Nagata dimension bound in the doubling and volume-doubling settings. Section~\ref{sec:asdim} establishes the polynomial-growth theorem and the polynomial control of the corresponding asymptotic-dimension function. Section~\ref{sec:rigidity} discusses nilmanifolds and the equality case in Theorem~\ref{thm:measure}. Finally, Appendix~\ref{sec:Gromov} contains the complementary simplicial argument and the proof of Proposition~\ref{prop:Euc}.

\section*{Acknowledgments}
Most of this work was carried out while JY was at the School of Mathematics, Georgia Institute of Technology, and XZ was a postdoctoral researcher at the Institute for Applied Mathematics, University of Bonn. XZ gratefully acknowledges financial support by the Deutsche Forschungsgemeinschaft (DFG) within the CRC 1060, at University of Bonn project number 211504053. XZ's work was also partially supported by the National Science Foundation Grant DMS-1928930 during his residency at the Simons Laufer Mathematical Sciences Institute (formerly MSRI) in Berkeley, California, in Fall 2024. The authors thank Anton Bernshteyn for valuable comments.

\section{Preliminaries}\label{sec:pre}

\subsection{Metric geometry}

    We collect some basic definitions. 
\begin{definition}
Let $(X,\dist)$ be a metric space. For $\varepsilon,\delta>0$, an \emph{$(\varepsilon,\delta)$-net} $T\subset X$ is a collection of points such that for every $x\in X$ there exists $y\in T$ with $x\in B_\varepsilon(y)$, and for distinct $y,z\in T$ we have $\dist(y,z)\ge \delta$.
\end{definition}

It follows from Zorn's lemma that for any $\varepsilon>0$ there exists an $(\varepsilon,\varepsilon)$-net in a metric space. If a metric space is separable, then there exists a countable $(\varepsilon,\varepsilon)$-net.

We now recall the notion of asymptotic dimension introduced by Gromov \cite{gromov1993asymptotic}*{\S 1.E}. A family $\mathcal U$ of subsets of a metric space $(X,\dist)$ is:
\begin{itemize}
    \item \emphdef{uniformly bounded} if $\sup_{U\in\mathcal U}\diam(U)<\infty$; more precisely, it is \emphdef{$D$-bounded} for some $D>0$ if $\sup_{U\in\mathcal U}\diam(U)\le D$;
    \item \emphd{$r$-disjoint} if $\dist(U,U')>r$ for all distinct $U,U'\in\mathcal U$.
\end{itemize}

\begin{definition}[$(r,D)$-covers]\label{def:cover}
For a metric space $(X,\dist)$, a tuple $(\mU_1,\dots,\mU_m)$ of $m$ families of subsets of $X$ is called an \emphdef{$(r,D)$-cover} of $X$ with $m$ \emphdef{layers} if:
\begin{enumerate}[label=\ep{\normalfont\arabic*}]
    \item\label{item:cover:rD} each $\mU_i$ is $r$-disjoint and $D$-bounded;
    \item\label{item:cover:cover} $\bigcup_{i=1}^m \mU_i$ is a cover of $X$.
\end{enumerate}
\end{definition}

\begin{definition}[Asymptotic dimension]\label{defn:ad}
Let $(X,\dist)$ be a metric space. The \emphd{asymptotic dimension} of $X$, denoted by $\asdim(X)$, is the minimum $n\in\N$ (if it exists) such that for every $r>0$ there exists an $(r,D(r))$-cover of $X$ with $n+1$ layers for some $D(r)<\infty$. We call $D(r)$ the \emphdef{control function}. If no such $n$ exists, we set $\asdim(X):=\infty$.
\end{definition}

There are more restrictive analogues of the asymptotic dimension, obtained by requiring linear growth of the control function.

\begin{definition}[Asymptotic Assouad--Nagata dimension]\label{defn:asand}
Let $(X,\dist)$ be a metric space. The \emphd{asymptotic Assouad--Nagata dimension} of $X$, denoted by $\asandim(X)$, is the minimum $n\in\N$ (if it exists) such that for every $r>0$ there exists an $(r,Cr+D)$-cover of $X$ with $n+1$ layers for some $C,D<\infty$. If no such $n$ exists, we set $\asandim(X):=\infty$.
\end{definition}

\begin{definition}[Assouad--Nagata dimension]\label{defn:andim}
Let $(X,\dist)$ be a metric space. The \emphd{Assouad--Nagata dimension} of $X$, denoted by $\andim(X)$, is the minimum $n\in\N$ (if it exists) such that for every $r>0$ there exists an $(r,Cr)$-cover of $X$ with $n+1$ layers for some uniform constant $C>0$. If no such $n$ exists, we set $\andim(X):=\infty$.
\end{definition}
 
\begin{remark}
It is immediate that
\[
\asdim(X)\le \asandim(X)\le \andim(X).
\]
\end{remark}

   % \begin{remark}
        %  Asymptotic dimension is invariant under quasi-isometry. Since any metric space is quasi-isometric to its metric completion, it is without loss of generality to restrict our attention to complete metric spaces when considering asymptotic dimension.
    %\end{remark}
    
For a noncompact metric measure space, there are several ways to quantify large-scale growth. For the purposes of this paper, we use two notions of growth. The first one makes sense for general metric spaces.

\begin{definition}\label{def:MetG}
Let $(X,\dist)$ be a metric space. For $r\ge 1$, the \emphd{metric growth function} is
\begin{equation}
\gamma(r)\defeq \sup\{|B_r(x)\cap T|:x\in X,\text{ $T$ is a $(1,1)$-net}\}.
\end{equation}
Whenever $\gamma(r)\in(0,\infty)$, define
\[
\rho(X,r)\defeq \frac{\log\gamma(r)}{\log(r+1)}.
\]
The \emphd{metric growth rate} of $X$ is
\[
\rho(X)\defeq \limsup_{r\to\infty}\rho(X,r).
\]
We say that $X$ has \emphd{polynomial metric growth} if $\rho(X)<\infty$.
\end{definition}

\begin{definition}\label{def:VolG}
Let $(X,\dist,\meas)$ be a metric measure space. For $r>0$, the \emphd{volume growth function} is
\[
V(r)\defeq \sup_{x\in X}\meas(B_r(x)).
\]
Whenever $V(r)\in(0,\infty)$, define
\[
\rho^V(X,r)\defeq \frac{\log V(r)}{\log(r+1)}.
\]
The \emphd{volume growth rate} of $X$ is
\[
\rho^V(X)\defeq \limsup_{r\to\infty}\rho^V(X,r).
\]
We say that $X$ has \emphd{polynomial volume growth} if $\rho^V(X)<\infty$.
\end{definition}

For a discrete metric space, for example the vertex set of a graph with the combinatorial metric, the counting measure identifies the two notions of growth. The relation between volume growth and metric growth in a more general setting is recorded in Lemma~\ref{lem:FiniteNet}.

We proceed to an equivalent formulation of large-scale dimension in terms of padded decompositions.

\begin{definition}\label{def:PadDecom}
    Let $(X,\dist,\meas)$ be a metric measure space, $r>0$. Fix an $(r,r)$-net $T\defeq\{x_i\}_{i\in \N}$ and parameters $R\ge r$, $m\in\N$, $D > 0$. An \emphd{$(R, D)$-padded decomposition with $m$ layers associated to $T$} is an $m$-tuple $(\mc{P}_1, \mc{P}_2, \ldots, \mc{P}_m)$ such that 
    \begin{enumerate}
        \item \label{item1:def:bpd} for each $i=1,\ldots, m$, $\mc{P}_i$ is a family of subsets of $X$ whose traces on $T$ partition $T$, i.e., $T \subseteq \bigcup \mc P_i$ and $C\,\cap\, C' \,\cap \,T=\0$ for distinct $C, C'\in \mc{P}_i$;
        \item \label{item2:def:bpd} for each $i=1,\ldots, m$, $\mc{P}_i$ is $D$-bounded;
        \item\label{item:Padded}  for any $x\in T$ there exist $i\in\{1,\dots,m\}$ and $C\in \mc P_i$ so that $B_R(x)\subset C$.
    \end{enumerate}
\end{definition}

\begin{lemma}\label{lem:Equivalence}
    Let $(X,\dist)$ be a complete separable metric space and $T\defeq\{x_i\}_{i\in \N}$ an $(r,r)$-net on $X$. The following holds for $R\ge r$.
    \begin{enumerate}
        \item\label{item:LemEqui1} If there exists a $(2R+r,D)$-cover of $X$ with $m$ layers, then there exists an $(R,2R+2r+D)$-padded decomposition with $m$ layers associated to $T$. %for any $\alpha$ such that $r^\alpha\ge 2r+1+D$.
        \item\label{item:LemEqui2} If there exists an $(R+2r,D)$-padded decomposition of $m$ layers associated to $T$, then there exists an $(R,D)$-cover of $X$ with $m$ layers. %for any $D$ such that $D\ge(r+2)^\alpha$. 
    \end{enumerate}
\end{lemma}
\begin{proof}
\eqref{item:LemEqui1}. Let $(\mc U_1,\dots,\mc U_m)$ be a $(2R+r,D)$-cover of $X$ with $m$ layers. For each $i=1,\dots,m$ and each $A\in\mc U_i$, set
\[
A^+\defeq B_R(A)=\bigcup_{x\in A} B_R(x).
\]
Let
\[
V_i\defeq \bigcup_{A\in\mc U_i} A^+
\qquad\text{and}\qquad
\mc P_i\defeq \{A^+:A\in\mc U_i\}\cup\{B_r(x_j):x_j\in T\setminus V_i\}.
\]
If $A,B\in\mc U_i$ are distinct, then $\dist(A^+,B^+)>r$ because $\dist(A,B)>2R+r$. Hence $A^+\cap B^+\cap T=\varnothing$. Moreover, if $x_j\in T\setminus V_i$, then $B_r(x_j)\cap T=\{x_j\}$ by the $r$-separation of $T$. Therefore the traces on $T$ of the sets in $\mc P_i$ partition $T$. Also,
\[
\diam(A^+)\le D+2R
\qquad\text{and}\qquad
\diam(B_r(x_j))\le 2r,
\]
so every element of $\mc P_i$ is $(2R+2r+D)$-bounded. Finally, if $x_j\in T$, choose $i$ and $A\in\mc U_i$ with $x_j\in A$. Then
\[
B_R(x_j)\subseteq A^+\in\mc P_i.
\]
Thus $(\mc P_1,\dots,\mc P_m)$ is an $(R,2R+2r+D)$-padded decomposition.

\eqref{item:LemEqui2}. Let $(\mc P_1,\dots,\mc P_m)$ be an $(R+2r,D)$-padded decomposition with $m$ layers associated to $T$. For each $i=1,\dots,m$ and each $A\in\mc P_i$, set
\[
A^\circ\defeq \{x\in A:\dist(x,A^c)\ge R+r\},
\]
and let
\[
\mc U_i\defeq \{A^\circ:A\in\mc P_i\}.
\]
We claim that $(\mc U_1,\dots,\mc U_m)$ is an $(R,D)$-cover of $X$.

To prove that it covers $X$, let $x\in X$. Choose $x_j\in T$ with $x\in B_r(x_j)$. By the padding property, there exist $i$ and $A\in\mc P_i$ such that
\[
B_{R+2r}(x_j)\subseteq A.
\]
Hence $x\in A^\circ\in\mc U_i$.

To prove that each $\mc U_i$ is $R$-disjoint, let $A^\circ,B^\circ\in\mc U_i$ be distinct and nonempty. Take $x\in A^\circ$ and $y\in B^\circ$. Choose $x_j\in T$ with $x\in B_r(x_j)$. Since $\dist(x,A^c)\ge R+r>r$, we have $x_j\in A\cap T$. Because the traces on $T$ of the sets in $\mc P_i$ are disjoint, $x_j\notin B$. Therefore $x_j\in B^c$, and hence
\[
\dist(y,x_j)\ge R+r.
\]
Since $\dist(x,x_j)<r$, we obtain
\[
\dist(x,y)\ge \dist(y,x_j)-\dist(x_j,x)\ge (R+r)-r=R.
\]
Thus $\dist(A^\circ,B^\circ)\ge R$, so each $\mc U_i$ is $R$-disjoint. The diameter bound is immediate from $A^\circ\subseteq A$ and $\diam(A)\le D$.
\end{proof}
It follows immediately that we can equivalently characterize the asymptotic dimension and Assouad--Nagata dimension by padded decompositions.

\begin{corollary}[Asymptotic and Assouad--Nagata dimension in terms of padded decomposition]\label{cor:PadtoBas}
    Let $(X,\d)$ be a complete separable metric space and $n \in \N$. Then $\asdim(X)\le n$ if and only if for every large enough $r$, there exists $\alpha>1$, such that $X$ admits an $(r, r^\alpha)$-padded decomposition with $n+1$ layers. Moreover, $\andim(X)\le n$ if and only if there exists $c>0$ such that for every $r>0$, $X$ admits an $(3r, cr)$-padded decomposition with $n+1$ layers.
\end{corollary}

Our goal will is to suitable padded decomposition. For this purpose we first associate a graph to a net in a metric space. 

\begin{definition}
    Given a complete{\if0 separable \fi} metric space $(X,\dist)$, parameters $M\in (r,\infty)$ and $r>0$, we take an $(r,r)$-net $T$ of $X$ and define a \emphd{net graph} $G^{M}(X,T)$ with vertex set $T$ and edge set $\{(x,y)\in T\times T\,\colon\ r\le\dist(x,y)\le M\}$. %in the way that there is an edge between $x_i$ and $x_j$ if and only if $1\le \dist(x_i,x_j)\le M$.
\end{definition}

 %It is clear that $G^{M}(X,T)$ is a Borel graph. %It is immediate that the 1-growth function $\gamma(r)$ is also the growth function of $G^{2M}(X)$. It is not hard to see that for a doubling metric space, $\gamma(r)=O(r^{\log_2 \cd})$. 
 The following lemma is classical and is essentially the fact that admitting a doubling measure implies metric doubling. We recall it here for convenience. %\JY{the lemma is the proof...?}

\begin{lemma}\label{lem:FiniteNet}
    Let $(X,\dist,\meas)$ be a metric measure space and $T$ be an $(r,r)$-net for some $r>0$. Then the following hold.
    \begin{enumerate}
        \item \label{item:LemF1} If $\meas$ is volume doubling with doubling constant $\cd$, then for any $R>r$ and $x\in X$, $$|B_R(x)\cap\, T|\le \cd^5 (R/r)^{\log_2 \cd }.$$ In particular, $\rho(X)\le\log_2 \cd$.
        \item \label{item:LemF2} If $(X,\dist,\meas)$ is proper and volume noncollapsed, and $\rho^V(X)<\infty$, then $\rho(X)\le\rho^V(X)$.
    \end{enumerate}
    %either \eqref{item:main 1} or \eqref{item:main 2} in Corollary \ref{thm:measure} holds. 
    % In particular any net graph $G^{2M}(X,T)$ has uniformly bounded degree.
\end{lemma}

\begin{proof}
    Write $T\defeq\{x_i\}_{i\in\N}$. Let $x\in X$ and $R\in(r,\infty)$. 
    
      We first claim that $B_R(x)\cap\, T$ is finite. If this is not the case, we assume $B_R(x)\cap\, T=\{y_j\}_{j\in \N}$ is infinite. Since for distinct $i$ and $j$, $\dist(y_i,y_j)\ge r$, we see that $\{B_{r/2}(y_j)\}_{j\in \N}$ are mutually disjoint and all lie inside $B_{R+r}(x)$. Hence \[
\sum_{j=1}^{\infty}\meas(B_{r/2}(y_j))\le \meas(B_{R+r}(x))<\infty,
\]
so in particular $\lim_{j\to\infty} \meas(B_{r/2}(y_j))=0$. 
      By the triangle inequality and the doubling condition, 
      $$
      \meas(B_{R+r}(x))\le \meas(B_{2R+r}(y_j))\le \cd^{\log_2(4R/r+2)+1  } \meas(B_{r/2}(y_j)).
      $$ 
      Letting $j\to\infty$ yields that $\meas(B_{R+r}(x))=0$, contradicting the full support of $\meas$.

   If instead \eqref{item:LemF2} holds, then $\ols B_R(x)$ is compact, and $T$ is discrete, so $B_R(x)\cap\, T$ is finite. 
   
   The claim is justified.
    
     Set $B_R(x)\cap T=\{y_1, y_2, \ldots, y_{\ell}\}$ for some $\ell\in \N$. As previously noticed, the balls $B_{r/2}(y_i)$ are mutually disjoint, and all lie inside $B_{R+r}(x)$. Let $y_j\in \mr{argmin} \{\meas(B_{r/2}(y_i)): y_i\in B_R(x)\cap T \}$. %\JY{The expression is not right: the argument is $i$ instead of $y_i$}
     It follows that 
    \begin{equation}\label{eq:DoublingBound}
            \ell\meas(B_{r/2}(y_j))\le \sum_{i=1}^\ell \meas(B_{r/2}(y_i))\le \meas(B_{R+r}(x)).
    \end{equation}
     \begin{itemize}
          \item  If \eqref{item:LemF1} is satisfied, then for the unique $n\in\N$ such that $2^{n-1}r<R \le 2^{n}r$, it follows from \eqref{eq:DoublingBound} that
    \[
    \ell \le \frac{\meas(B_{R+r/2}(x))}{\meas(B_{r/2}(y_j))}\le \frac{\meas(B_{2R+r}(y_j))}{\meas(B_{r/2}(y_j))}\le \cd^{\log_2(4R/r+2)+1}\le \cd^{n+4}\le \cd^5 (R/r)^{\log_2 \cd}.
    \]
    In particular, taking $r=1$ and take supremum over all $(1,1)$-nets, we get that $\rho(X)\le\log_2 \cd$.
    
        \item If \eqref{item:LemF2} holds, fix any $b>\rho^V(X)$. By definition of $\rho^V(X)$, for all sufficiently large $R$ we have
\[
V(R+1)\le (R+2)^b.
\]
Using \eqref{eq:DoublingBound} and the definition of $v$,
\[
\ell\le \frac{\meas(B_{R+1}(x))}{\meas(B_{1/2}(y_j))}
\le \frac{1}{v}(R+2)^b
\]
for all sufficiently large $R$. Taking the supremum over all $(1,1)$-nets shows that $\rho(X)\le b$. Since $b>\rho^V(X)$ is arbitrary, we conclude that $\rho(X)\le \rho^V(X)$.
    %\JY{what is $v$?}
      \end{itemize}
    %If \eqref{item:main 3} is satisfied, then it follows from \eqref{eq:DoublingBound} and Bishop-Gromov inequality that
    %\[
    %\ell \le \frac{\meas(B_{r+\frac12}(p))}{B_{\frac12}(y_j)}\le \frac{\meas(B_{2r+1}(y_i))}{B_{\frac12}(y_j)}\le (4r+2)^N.
    %\]
    In both cases, the estimates are independent of the net $T$.  
    %Finally, to see $G^{2M}(X)$ has uniformly bounded degree it suffices to take $r=2M$, since the point $p\in X$ and the net $T$ are both arbitrary. The proof is completed.
\end{proof}

\begin{remark}
    Lemma \ref{lem:FiniteNet} can replace \cite{papasoglu2021polynomial}*{Corollary 3.3}, and the bounded geometry assumption there can be weakened to either volume doubling, volume noncollapsed, or the bounded geometry condition in the sense of metric space as in \cite{YuNovikov}. 
\end{remark}

  An immediate consequence of Lemma \ref{lem:FiniteNet} is the following, obtained by taking $R=M$ in the proof. 
\begin{corollary}\label{cor:UniBddDeg}
      Let $(X,\dist,\meas)$ be a metric measure space and $T$ be an $(r,r)$-net on $X$ for some $r>0$. Assume either:
      \begin{enumerate}
        \item \label{item:LemU1} $\meas$ is volume doubling with doubling constant $\cd$; or
        \item \label{item:LemU2} $(X,\dist,\meas)$ is proper, volume noncollapsed, polynomial volume growth, and $r=1$. %with $\rho^V(X)\le b$ for some $b>0$.
    \end{enumerate}
      Then for each $M>r$, the net graph $G^{M}(X,T)$ has uniformly bounded degree independent of $T$.
\end{corollary}

In a similar fashion we have an analogue for doubling metric spaces. 

\begin{lemma}\label{lem:DMgrowth}
    Let $(X,\dist)$ be a doubling metric space with doubling constant $N$, and let $T$ be an $(r,r)$-net for some $r>0$. Then for any $R>r$, we have $$|  B_R(x)\cap T|\le N^2 (R/r)^{\log_2 N}.$$ 
    In particular, $\rho(X)\le \log_2 N$.
\end{lemma}

\begin{proof}
     For every $R > r$, there exists a unique positive integer $n$ such that $2^{n-1}r < R \leq 2^n r$. For any $x \in X$, the doubling condition implies that $B_R(x)$ can be covered by at most $N^{n+1}$ balls of radius $r/2$. Clearly, each ball of radius $r/2$ can contain at most one element of $T$. Therefore, 
    \[
    | B_R(x) \cap T | \leq N^{n+1} \leq N^2 \cdot N^{\log_2 (R/r)} = N^2 \left(\frac{R}{r}\right)^{\log_2 N}.
    \]
    This estimate is independent of $x$ and $T$, so the final conclusion follows from the definition of $\rho(X)$.
\end{proof}

Again, taking $R=M$ in the previous proof we immediately obtain the following.

\begin{corollary}\label{cor:DMBddDeg}
     Let $(X,\dist)$ be a doubling metric space and $T$ be an $(r,r)$-net on $X$ for any $r>0$. Then for each $M>r$, the net graph $G^M(X,T)$ has uniformly bounded degree independent of $T$.
\end{corollary}

%\JY{your definition of the net graph relies on the net, so the statement here is still unclear}

    \subsection{The Lov\'asz Local Lemma}\label{sec:LLL}
     
     We will employ the Lov\'asz Local Lemma (the \emphdef{LLL} for short) in the setting of \emph{constraint satisfaction problems}. The LLL, introduced by Erd\H{o}s and Lov\'asz in 1975, is a powerful probabilistic tool. The LLL is not only widely used in combinatorics but has also recently found numerous applications in other areas, such as topological dynamics, ergodic theory, descriptive set theory, and more.

     %We borrow the following notation and terminology from \cite{BerDist}.
    \begin{definition}
     Fix a set $X$ and a compact Polish space $Y$ equipped with a Borel probability measure $\lambda$.
        \begin{itemize}[wide]
    \item A $Y$-\emphdef{coloring} of a set $S$ is a function $f \colon S \rightarrow Y$.
     \item Given a finite set $D \subseteq X$, an \emphdef{$X$-constraint} (or simply a \emphdef{constraint} if $X$ is clear from the context) with \emphdef{domain} $D$ is a measurable set $A \subseteq Y^D$ of colorings of $D$. We write $\dom(A) \defeq D$. 
     \item A coloring $f \colon X \rightarrow Y$ of $X$ \emphdef{violates} a constraint $A$ with domain $D$ if the restriction of $f$ to $D$ is in $A$, and \emphdef{satisfies} $A$ otherwise.
     \item A \emphdef{constraint satisfaction problem} (or a \emphdef{CSP} for short) $\A$ on $X$ with range $Y$, in symbols $\A \colon X \rightarrow^{?} Y$, is a set of $X$-constraints. 
     \item A \emphdef{solution} to a CSP $\A \colon X \rightarrow^{?} Y$ is a coloring $f \colon X\rightarrow Y$ that satisfies every constraint $A \in \A$.
     \end{itemize}
    
     \end{definition}

 Fix a CSP $\A \colon X \rightarrow^{?} Y$. 
  Recall that $Y$ is a compact Polish space equipped with the Borel probability measure $\lambda$. 
     For any finite set $D \subseteq X$, $Y^D$ is equipped with the probability measure $\lambda^D$. 
     For each measurable constraint $A \in \A$, the \emphdef{probability} $\P[A]$ of $A$ is defined as the probability that $A$ is violated by a random coloring $f \colon X \to Y$, that is,
 \[\P[A] \,\defeq\, \lambda^{\dom(A)}(A).\]
 The \emphdef{neighborhood} of $A$ in $\A$ is the set 
 \[N(A) \,\defeq\, \set{A' \in \A \,:\, A' \neq A  \text{ and } \dom(A') \cap \dom(A) \neq \0}.\]
 Let $\p(\B) \defeq \sup_{A \in \B} \P[A]$ and $\d(\B) \defeq \sup_{A \in \B} |N(A)|$.

\begin{theorem}[{Lov\'asz Local Lemma \cites{EL, SpencerLLL}}]
\label{thm: LLL}
If $\A$ is a CSP such that $e \cdot \p(\A) \cdot (\d(\A) + 1) < 1$,
then $\A$ has a solution. \ep{Here $e = 2.71\ldots$ is the base of the natural logarithm.}
\end{theorem}

\begin{remark}\label{remk:LLL-probability-space}
The form of the Lov\'asz Local Lemma stated above is used with respect to product probability measures. In particular, the range space need not be finite and the distribution need not be uniform. This will be important in Section~\ref{sec:Nagata}, where the random radii are sampled from a truncated exponential distribution on an interval, and in Section~\ref{sec:asdim}, where they are sampled from a truncated geometric distribution. The probability $\P[A]$ of a constraint is always understood as the probability, under the corresponding product measure on its finite domain, that the random coloring violates $A$.
\end{remark}

\begin{remark}
The LLL is often stated for finite $\A$. However, Theorem \ref{thm: LLL} holds for infinite $\A$ as well.
Given any CSP $\A$ such that $e \cdot \p(\A) \cdot (\d(\A) + 1) < 1$, for any $\epsilon \in (0, 1)$ we can construct an $\A'$ as follows: For each $A \in \A$ we pick an open set $A'$ such that $A \subseteq A'$ and $\P[A] \ge (1-\epsilon)\P[A']$ by regularity of measure \cite{KechrisDST}*{Theorem 17.10}. Then we can construct a new CSP $\A' \defeq \set{A'\,:\,A \in \A}$.
We can pick $\epsilon$ such that $e \cdot \p(\A') \cdot (\d(\A') + 1) < 1$.        Assume there exists no solution to $\A'$, then we have $Y^X = \bigcup \A'$. By Tychonoff's theorem $Y^X$ is compact, hence there exists some finite subset $\A'' \subseteq \A'$ such that $Y^X = \bigcup \A''$, which implies there exists no solution to the CSP $\A''$. Notice that $e \cdot \p(\A'') \cdot (\d(\A'') + 1) < 1$, then there should be some solution to $\A''$, a contradiction. Hence there exists some solution to $\A'$, which is also a solution to $\A$. 
\end{remark}

\subsection{A ball-carving scheme}\label{subsec:ball-carving}
Before discussing the ball-carving scheme, we introduce some basic knowledge in graph theory. 

Given a graph $G$, a \emphd{$k$-coloring} is a function $c: V(G) \to S$, where $S$ is a set of size $k$. 
A $k$-coloring $c$ is \emphd{proper} if $c(u) \neq c(v)$ for all $uv \in E(G)$. The \emphd{chromatic number} $\chi(G)$ is the least $k \in \N$ such that $G$ has a proper $k$-coloring. A set $I \subseteq V(G)$ is \emphd{independent} in $G$ if $uv \not \in E(G)$ for all $u$, $v \in I$. Note that if $c \colon V(G) \to S$ is a proper $k$-coloring of $G$, then $V(G) = \bigcup_{s \in S} c^{-1}(s)$ is a partition of $V(G)$ into $k$ independent sets. The \emphd{greedy coloring} with respect to a vertex ordering $v_1, v_2, \ldots$ colors vertices in the order $v_1, v_2, \ldots$, assigning to $v_i$ the least-indexed color not used on its neighbors earlier. Let $\Delta(G)$ denote the maximum degree of $G$. If $\Delta(G) < \infty$, then $\chi(G) \le \Delta(G)+1$ because each vertex has at most $\Delta(G)$ earlier neighbors in a vertex coloring, so the greedy coloring will not use more than $\Delta(G)+1$ colors. 

Now we are ready to discuss the ball-carving scheme. Let $(X,\dist)$ be a complete separable metric space, let $r>0$, and let $T$ be an $(r,r)$-net in $X$. Assume that the net graph $G^{2M}(X,T)$ has finite maximum degree. Fix a proper coloring
\[
T = I_0\sqcup I_1\sqcup \cdots \sqcup I_{k-1}
\]
of $G^{2M}(X,T)$. Then $\dist(x,y)>2M$ whenever $x,y\in I_i$ are distinct.
Given a function $t:T\to [0,M]$, ddefine inductively families of balls $\mc C_0,\ldots, \mc C_{k-1}$ as follows:
\begin{equation*}
    \mc C_0\defeq \{B_{t(x)}:x\in I_0\},\quad \mc C_{i+1}=\left\{B_{t(x)}(x)\setminus \bigcup_{j=0}^{i}\left(\bigcup \mc C_j\right):x\in I_{i+1}\right\}.
\end{equation*}
Set
\[
\mc P_t:=\bigcup_{i=0}^{k-1}\mc C_i.
\]
By construction, every element of $\mc P_t$ has diameter at most $2M$. Moreover, the traces on $T$ of the sets in $\mc P_t$ partition $T$. Thus $\mc P_t$ is a $2M$-bounded family of subsets of $X$ associated to the net $T$.

\section{Doubling metric spaces and Assouad--Nagata dimension}\label{sec:Nagata}

Let $(X,\dist)$ be a doubling metric space with doubling constant $N$. Given $r>0$, fix an $(r,r)$-net $T=\{x_i\}_{i=0}^{\infty}$ in $(X,\dist)$. Recall from Lemma~\ref{lem:DMgrowth} that each ball of radius $R$ contains at most $N^2(R/r)^{\log_2 N}$ points of $T$. We will frequently use this estimate.

We start by adapting the randomized ball carving construction for graphs to the metric space setting with the help of a net graph.

\subsection{A randomized ball-carving construction}

Let $(X,\dist)$ be a complete separable metric space and let $T$ be an $(r,r)$-net on $X$ for some $r>0$. To find a padded decomposition of $X$ with the desired properties, we employ a variant of the randomized ball-carving construction from \cite{BY23}, originating in computer science.

Fix parameters $0<l<M$. For each function $t:T\to [l,M]$, the ball-carving scheme from Section~\ref{subsec:ball-carving} produces a $2M$-bounded family $\mc P_t$ of subsets of $X$ whose traces on $T$ partition $T$. Given a positive integer $m$, let
\[
\bm t=(t_1,\dots,t_m),\qquad t_i:T\to [l,M].
\]
Then $\bm t$ gives rise to an $m$-tuple of families
\[
(\mc P_{t_1},\dots,\mc P_{t_m}).
\]
Our goal is to choose $\bm t$ so that this becomes a $(3r,cr)$-padded decomposition for a suitable constant $c=c(N)$.

To this end, we let the coordinates $t_1,\dots,t_m$ be i.i.d. random functions, each with values distributed according to the \emphd{truncated exponential distribution} on $[l,M]$ with parameter $\lambda$, denoted by $\mr{Texp}(\lambda,M,l)$. Its density is
\[
\mb P[t=z]=
\begin{cases}
\dfrac{\lambda e^{-\lambda z}}{e^{-\lambda l}-e^{-\lambda M}}, & z\in [l,M],\\[1ex]
0, & \text{otherwise}.
\end{cases}
\]
We begin with some elementary estimates for $\mr{Texp}(\lambda,M,l)$.

\begin{lemma}\label{lem:Estimate}
Let $t\sim \mr{Texp}(\lambda,M,l)$ and assume that $(M-l)\lambda\ge 2$ and $l\lambda\le 1$. Then:
\begin{enumerate}
    \item\label{item:estimate 1} For every $\beta\in [l,M]$,
    \[
    \mb P[t\ge \beta]\le 4e^{-\lambda\beta}.
    \]
    \item\label{item:estimate 2} For all $\alpha,\beta\ge 0$ such that $\alpha\ge l$, $\alpha\le M/2$, and $\alpha+\beta<M$,
    \[
    \mb P[t\le \alpha+\beta\,|\,t\ge \alpha]\le 2\lambda\beta.
    \]
\end{enumerate}
\end{lemma}

\begin{proof}
This follows from straightforward computations.

For \eqref{item:estimate 1},
    \begin{align*}
        \mb P[t\ge \beta]&=\int_{\beta}^M \frac{\lambda e^{-\lambda t}}{e^{-\lambda l}-e^{-\lambda M}}\,\mathrm dt = \frac{e^{\lambda M}}{e^{\lambda (M-l)}-1}(e^{-\lambda \beta}-e^{-\lambda M}) \\&= \frac{e^{\lambda l}}{1-e^{-\lambda (M-l)}}(e^{-\lambda \beta}-e^{-\lambda M})  \le \frac{e^1}{1-e^{-2}}e^{-\lambda \beta}  < 4e^{-\lambda \beta}.
    \end{align*}

For \eqref{item:estimate 2},
    \begin{align*}
    \mb P[t \le \alpha+\beta\,|\,t \ge \alpha]&=\frac{\int_{\alpha}^{\alpha+\beta}e^{-\lambda t}\,\mathrm dt}{\int_{\alpha}^{M}e^{-\lambda t}\,\mathrm dt} = \frac{e^{-\lambda \alpha}-e^{-\lambda (\alpha+\beta)}}{e^{-\lambda \alpha}-e^{-\lambda M}} = \frac{1-e^{-\lambda \beta}}{1-e^{-\lambda (M-\alpha)}} \\
    &\le \frac{\lambda\beta}{1-e^{-\lambda M/2}} \le \frac{\lambda\beta}{1-e^{-1}} <2\lambda\beta.\qedhere
    \end{align*}
\end{proof}

Some estimates are needed for the ball carving construction. We say for $u\in T$ and $r>0$, a ball $B_r(u)$ is \emphd{cut} if it intersects two distinct $C_i$ and $C_j$ in $\mc P_t$.
We say that a ball $B_r(u)$ is \emphd{cut} if it intersects two distinct clusters of $\mc P_t$. This is a scenario we wish to avoid as it violates \eqref{item:Padded} of Definition \ref{def:PadDecom}. 
The following probability estimates will be the crucial for applying the Lov\'asz Local Lemma.

\begin{lemma}\label{lem:cut_N}
Let $(X,\dist)$ be a doubling metric space with doubling constant $N$, let $r>0$, and let $T$ be an $(r,r)$-net of $X$. Fix $u\in T$. Take parameters $\eps\in(0,1/3]$ and $D>1/\eps+1/2$, and set
\[
\lambda=\frac{\eps}{r},\qquad M=(2D+3)r,\qquad l=3r.
\]
Let $t:T\to [l,M]$ be sampled according to $\mr{Texp}(\lambda,M,l)$. Then
\[
\mb P[B_{3r}(u)\text{ is cut}\,]
\le
4N^3(D+3)^{\log_2 N}e^{-(D-\frac32)\eps}+12\eps.
\]
\end{lemma}
   Note that $G^{2M}(X,T)$ has uniformly bounded degree thanks to the doubling condition, so the randomized ball carving construction can be carried out. 
\begin{proof}
First we check the assumptions of Lemma~\ref{lem:Estimate}. Since
\[
(M-l)\lambda = \bigl((2D+3)r-3r\bigr)\frac{\eps}{r}=2D\eps\ge 2
\]
and
\[
l\lambda = 3r\cdot \frac{\eps}{r}=3\eps\le 1,
\]
Lemma~\ref{lem:Estimate} applies.

Let $I_0,\dots,I_{k-1}$ be the color classes coming from a proper coloring of the net graph $G^{2M}(X,T)$. For $x,y\in T$, write $y\prec x$ if $y\in I_i$ and $x\in I_j$ with $i<j$.

Set
\[
B:=B_{3r}(u)
\qquad\text{and}\qquad
B_x:=B_{t(x)}(x)\quad \text{for }x\in T.
\]
We say that $B$ is \emphd{cut by $B_x$} if
\[
B_y\cap B=\varnothing \quad\text{for all } y\prec x,
\]
and
\[
\varnothing \neq B_x\cap B \neq B.
\]
Clearly, $B$ is cut if and only if it is cut by some $B_x$.

Define $A_{\mathrm{far}}$ to be the event that there exists some $B_x$ cutting $B$ with $\dist(x,u)\ge M/2$, and define $A_{\mathrm{near}}$ to be the event that there exists some $B_x$ cutting $B$ with $\dist(x,u)<M/2$. Then
\[
\mb P[B_{3r}(u)\text{ is cut}]
\le
\mb P[A_{\mathrm{far}}]+\mb P[A_{\mathrm{near}}].
\]

We first estimate $\mb P[A_{\mathrm{far}}]$. Fix $x\in T$ with $\dist(x,u)\ge M/2$. If $B_x$ cuts $B$, then necessarily
\[
\dist(x,u)\le t(x)+3r,
\]
hence
\[
\mb P[B_x\text{ cuts }B]
\le
\mb P[t(x)\ge M/2-3r]
\le
4e^{-\lambda(M/2-3r)}
=
4e^{-(D-\frac32)\eps},
\]
where the second inequality follows from Lemma~\ref{lem:Estimate}\eqref{item:estimate 1}.

Moreover, if $B_x$ cuts $B$, then
\[
\dist(x,u)\le t(x)+3r \le M+3r=(2D+6)r.
\]
Hence the number of $x\in T$ that can possibly cut $B$ is bounded above by
\[
N^2(2D+6)^{\log_2 N}
=
N^3(D+3)^{\log_2 N},
\]
using Lemma~\ref{lem:DMgrowth}. Therefore,
\[
\mb P[A_{\mathrm{far}}]
\le
4N^3(D+3)^{\log_2 N}e^{-(D-\frac32)\eps}.
\]

Next we estimate $\mb P[A_{\mathrm{near}}]$. On the finite set $T\cap B_{M/2}(u)$, the partial order $\prec$ is actually total, because two distinct points in $T\cap B_{M/2}(u)$ cannot belong to the same color class. Let $Y$ be the $\prec$-smallest element of $T\cap B_{M/2}(u)$ such that
\[
B_Y\cap B\neq \varnothing.
\]
This is well-defined since $u\in T\cap B_{M/2}(u)$ and $B_u\cap B\neq \varnothing$.

For each $x\in T\cap B_{M/2}(u)$, the event $Y=x$ is equivalent to the conjunction of:
\begin{itemize}
    \item $t(x)\ge \dist(u,x)-3r$;
    \item $t(y)<\dist(u,y)-3r$ for all $y\prec x$ in $T\cap B_{M/2}(u)$.
\end{itemize}
If $A_{\mathrm{near}}$ occurs, then $B_Y$ cuts $B$. In particular, if
\[
t(Y)\ge \dist(u,Y)+3r,
\]
then $B_Y\supseteq B$ and $A_{\mathrm{near}}$ cannot occur. Therefore,
\begin{align*}
\mb P[A_{\mathrm{near}}\,|\,Y=x]
&\le
\mb P[t(x)<\dist(u,x)+3r\,|\,Y=x] \\
&=
\mb P[t(x)<\dist(u,x)+3r\,|\,t(x)\ge \dist(u,x)-3r] \\
&\le
2\lambda\cdot 6r \\
&=
12\eps,
\end{align*}
where the inequality follows from Lemma~\ref{lem:Estimate}\eqref{item:estimate 2} with $\beta=6r$.

Summing over all possible values of $Y$, we obtain
\[
\mb P[A_{\mathrm{near}}]
=
\sum_{x\in T\cap B_{M/2}(u)}
\mb P[Y=x]\,
\mb P[A_{\mathrm{near}}\,|\,Y=x]
\le
12\eps.
\]

Combining the two estimates gives
\[
\mb P[B_{3r}(u)\text{ is cut}]
\le
4N^3(D+3)^{\log_2 N}e^{-(D-\frac32)\eps}+12\eps.\qedhere
\]
\end{proof}

We now reformulate the existence of padded decompositions as a CSP to which the Lov\'asz Local Lemma can be applied. Given $m\in \N^+$, define a CSP
\[
\mc A_m: T \to^? [l,M]^m
\]
as follows. For each $u\in T$, let $A_{u,m}$ be the constraint with domain
\[
B_{M+3r}(u)\cap T
\]
such that a function
\[
\bm t=(t_1,\dots,t_m):T\to [l,M]^m
\]
satisfies $A_{u,m}$ if and only if $B_{3r}(u)$ is \emph{not} cut in at least one of the partitions
\[
\mc P_{t_1},\dots,\mc P_{t_m}.
\]
Then
\[
\mc A_m:=\{A_{u,m}:u\in T\}
\]
is a CSP. We equip $[l,M]^m$ with the product measure of $m$ copies of $\mr{Texp}(\lambda,M,l)$, that is, the coordinates $t_1,\dots,t_m$ are taken to be i.i.d. random variables with distribution $\mr{Texp}(\lambda,M,l)$.

\begin{lemma}\label{lem:CSP_N}
Under the assumptions of Lemma~\ref{lem:cut_N}, the CSP defined above satisfies
\[
\p(\mc A_m)\le
\left(
4N^3(D+3)^{\log_2 N}e^{-(D-\frac32)\eps}+12\eps
\right)^m
\]
and
\[
\d(\mc A_m)\le N^4(D+3)^{\log_2 N}-1.
\]
\end{lemma}

\begin{proof}
The bound for $\p(\mc A_m)$ follows immediately from Lemma~\ref{lem:cut_N}.

We now estimate $\d(\mc A_m)$. If
\[
\bigl(B_{M+3r}(u)\cap T\bigr)\cap \bigl(B_{M+3r}(v)\cap T\bigr)\neq \0,
\]
then necessarily
\[
\dist(u,v)\le 2M+6r.
\]
Hence, by Lemma~\ref{lem:DMgrowth},
\begin{align*}
\d(\mc A_m)
&\le
N^2\left(\frac{2M+6r}{r}\right)^{\log_2 N}-1 \\
&=
N^2(4D+12)^{\log_2 N}-1 \\
&=
N^4(D+3)^{\log_2 N}-1.
\end{align*}
Here the term $-1$ excludes the case $u=v$.
\end{proof}

\subsection{Assouad--Nagata dimension of doubling metric spaces}

\begin{theorem}\label{thm:double-pad}
Let $(X,\dist)$ be a doubling metric space with doubling constant $N$. Then there exists a constant $c=c(N)>0$ such that for every $r>0$, $X$ admits a $(3r,cr)$-padded decomposition with $\floor{\log_2 N}+1$ layers.
\end{theorem}

\begin{proof}
Let
\[
b:=\log_2 N,\qquad m:=\floor{b}+1.
\]
Choose parameters as in Lemma~\ref{lem:cut_N},
\[
\lambda=\frac{\eps}{r},\qquad M=(2D+3)r,\qquad l=3r.
\]
Let $\mc A_m$ be the CSP constructed above. By Theorem~\ref{thm: LLL}, the CSP $\mc A_m$ has a solution provided that
\[
e\,
\left(
4N^3(D+3)^b e^{-(D-\frac32)\eps}+12\eps
\right)^m
\left(
N^4(D+3)^b
\right)
<1.
\]
Equivalently,
\[
\left(
4N^3(D+3)^b e^{-(D-\frac32)\eps}+12\eps
\right)^m
(D+3)^b
<
16^{-b}e^{-1}.
\]
Taking the $m$-th root gives
\[
4N^3(D+3)^{b(m+1)/m} e^{-(D-\frac32)\eps}
+
12\eps (D+3)^{b/m}
<
8^{-b/m}e^{-1/m}.
\]
Since $b/m<1$, choose $\alpha\in(0,1)$ such that
\[
\frac{b}{m}+\alpha<1,
\]
and set
\[
\eps=(D+3)^{-b/m-\alpha}.
\]
Then, as $D\to\infty$, the left-hand side tends to $0$, while the right-hand side is independent of $D$. Hence, for $D$ sufficiently large, the Lov\'asz Local Lemma applies.

We obtain a solution
\[
\bm t=(t_1,\dots,t_m):T\to [l,M]^m.
\]
By construction, the corresponding tuple
\[
(\mc P_{t_1},\dots,\mc P_{t_m})
\]
satisfies the padding condition in Definition~\ref{def:PadDecom}. Moreover, each element of each $\mc P_{t_i}$ has diameter at most $2M=(4D+6)r$. Thus, if we set
\[
c:=4D+6,
\]
then $(\mc P_{t_1},\dots,\mc P_{t_m})$ is a $(3r,cr)$-padded decomposition with $m=\floor{\log_2 N}+1$ layers.
\end{proof}

As a direct consequence, we obtain an upper bound the Assouad--Nagata dimension of doubling metric spaces: 
\begin{corollary}[Theorem~\ref{thm:VD}]
Every doubling metric space $(X,\dist)$ with doubling constant $N$ satisfies
\[
\andim(X)\le \floor{\log_2 N}.
\]
\end{corollary}

\begin{proof}
By Theorem~\ref{thm:double-pad} and Corollary~\ref{cor:PadtoBas}.
\end{proof}

\section{Metric measure spaces of polynomial growth and asymptotic dimension}\label{sec:asdim}

\subsection{A randomized ball-carving construction}

Let $(X,\dist)$ be a complete separable metric space and let $T$ be a $(1,1)$-net in $X$. Fix an integer $M>1$. For each function $t:T\to \{1,\ldots,M\}$, the ball-carving scheme from Section~\ref{subsec:ball-carving} produces a $2M$-bounded family $\mc P_t$ of subsets of $X$ whose traces on $T$ partition $T$.

Given $m\in \N^+$, let
\[
\mathbf t=(t_1,\ldots,t_m),\qquad t_i:T\to \{1,\ldots,M\}.
\]
Then $\mathbf t$ gives rise to an $m$-tuple
\[
(\mc P_{t_1},\ldots,\mc P_{t_m}).
\]
Our goal is to choose $\mathbf t$ so that this becomes an $(r,r^\alpha)$-padded decomposition for suitable parameters $r>1$, $\alpha>1$, and $m\in\N$.

To this end, we let the coordinates $t_1,\ldots,t_m$ be i.i.d. random variables distributed according to the \emphd{truncated geometric distribution} on $\{1,\ldots,M\}$ with parameter $p\in(0,1)$, denoted by $\mr{Tgeo}(p,M)$. Its density is given by
\[
\mb P[t=n]=
\begin{cases}
p(1-p)^{n-1}, & n=1,2,\ldots,M-1,\\
(1-p)^{M-1}, & n=M.
\end{cases}
\]
We will need the following elementary estimates.

\begin{lemma}\label{lem:Est}
For a random variable $t\sim\mr{Tgeo}(p,M)$, we have: 
\begin{enumerate}
    \item\label{item:esti 1} For every integer $n\in\{1,\ldots,M\}$,
    \[
    \mb P[t\ge n]=(1-p)^{n-1}.
    \]
    \item\label{item:esti 2} For all integers $m,n\ge 1$ such that $m+n<M$,
    \[
    \mb P[t\le m+n\,|\,t\ge m]=1-(1-p)^{n+1}.
    \]
\end{enumerate}
\end{lemma}

\begin{proof}
This is a straightforward computation.

For \eqref{item:esti 1},
\[
\mb P[t\ge n]
=1-\sum_{i=1}^{n-1}p(1-p)^{i-1}
=(1-p)^{n-1}.
\]

For \eqref{item:esti 2},
\[
\mb P[t\le m+n\,|\,t\ge m]
=
\frac{\mb P[m\le t\le m+n]}{\mb P[t\ge m]}
=
\frac{(1-p)^{m-1}-(1-p)^{m+n}}{(1-p)^{m-1}}
=
1-(1-p)^{n+1}.
\qedhere
\]
\end{proof}

In the next lemma, we estimate the probability that a ball is cut under the truncated geometric distribution; this estimate will later be made sufficiently small for an application of the Lovász Local Lemma.

\begin{lemma}\label{lem:cut}
Let $(X,\dist)$ be a complete separable metric space, let $T$ be a $(1,1)$-net of $X$, and let
\[
t:T\to \{1,\ldots,M\}
\]
be a random function whose values are i.i.d. with distribution $\mr{Tgeo}(p,M)$. Fix $u\in T$, $r\ge 9$, and $b\ge 0$. Assume that
\[
p\le \frac{1}{4b+9}
\quad\text{ and }\quad
M=\left\lfloor 4(b+1)\frac{1}{p}\ln\frac{1}{p}\right\rfloor.
\]
If $\gamma(s)\le s^b$ for all $s\ge r$, then
\[
\mb P[B_r(u)\text{ is cut}\,]\le 20rp.
\]
\end{lemma}

Note that the assumption on $\gamma$ implies that the net graph $G^{2M}(X,T)$ has finite maximum degree, so the randomized ball-carving construction is well-defined.

\begin{proof}
We may assume that $rp\le 1$, since otherwise the statement is trivial. Also,
\[
\mb P[B_r(u)\text{ is cut}]
\le
\mb P[B_{\lceil r\rceil}(u)\text{ is cut}],
\]
and $\lceil r\rceil\le r+1\le 10r/9$ when $r\ge 9$. Thus it suffices to prove the estimate for integer $r\ge 9$, so from now on we assume that $r$ is an integer.

Let $I_0,\ldots,I_{k-1}$ be the color classes of a proper coloring of $G^{2M}(X,T)$. For $x,y\in T$, write $y\prec x$ if $y\in I_i$ and $x\in I_j$ with $i<j$.

Set
\[
B:=B_r(u),
\qquad
B_x:=B_{t(x)}(x)\quad\text{for }x\in T.
\]
We say that $B$ is \emphd{cut by $B_x$} if
\[
B_y\cap B=\0 \quad\text{for all }y\prec x,
\]
and
\[
\0\neq B_x\cap B\neq B.
\]
Clearly, $B$ is cut if and only if it is cut by some $B_x$.

Define $A_{\mathrm{far}}$ to be the event that there exists some $B_x$ cutting $B$ with
\[
\dist(x,u)\ge M-r,
\]
and define $A_{\mathrm{near}}$ to be the event that there exists some $B_x$ cutting $B$ with
\[
\dist(x,u)<M-r.
\]
Then
\[
\mb P[B_r(u)\text{ is cut}]
\le
\mb P[A_{\mathrm{far}}]+\mb P[A_{\mathrm{near}}].
\]

We first estimate $\mb P[A_{\mathrm{far}}]$. Fix $x\in T$ with $\dist(x,u)\ge M-r$. If $B_x$ cuts $B$, then necessarily
\[
\dist(x,u)\le t(x)+r,
\]
hence
\[
\mb P[B_x\text{ cuts }B]
\le
\mb P[t(x)\ge M-2r]
\le
(1-p)^{M-2r-1}.
\]
Since $r\le 1/p$, we obtain
\[
(1-p)^{M-2r-1}
\le
(1-p)^{-2/p-1}(1-p)^M.
\]
The function $z\mapsto (1-z)^{-2/z-1}$ is increasing on $(0,1)$, and
\[
p\le \frac{1}{4b+9}\le \frac15,
\]
so
\[
(1-p)^{-2/p-1}\le \left(\frac54\right)^{11}.
\]
Using $1-p\le e^{-p}$, we get
\[
(1-p)^M
\le
e^p\,p^{4b+4}
\le
e^{1/5}p^{4b+4}.
\]
Combining these estimates yields
\[
\mb P[B_x\text{ cuts }B]
\le
\left(\frac54\right)^{11}e^{1/5}p^{4b+4}
\le
15\,p^{4b+4}.
\]

If $B_x$ cuts $B$, then
\[
\dist(x,u)\le t(x)+r\le M+r,
\]
so the number of points $x\in T$ that can possibly cut $B$ is bounded by
\[
\gamma(M+r)\le (M+r)^b \le \left(\frac{4(b+1)\ln(1/p)+1}{p}\right)^b.
\]
Since $\ln(1/p)<1/p$ and $p\le 1/(4b+9)$, we have
\[
4(b+1)\ln(1/p)+1
\le
\frac{4(b+1)}{p}+1
\le
\frac{4b+5}{p}
\le
p^{-2}.
\]
Hence
\[
\gamma(M+r)\le p^{-3b}.
\]
Therefore,
\[
\mb P[A_{\mathrm{far}}]
\le
15\,p^{4b+4}\cdot p^{-3b}
=
15\,p^{b+4}
\le
15p.
\]

Next we estimate $\mb P[A_{\mathrm{near}}]$. On the finite set $T\cap B_{M-r}(u)$, the relation $\prec$ is in fact a total order, since two distinct points in $T\cap B_{M-r}(u)$ cannot belong to the same color class. Let $Y$ be the $\prec$-smallest element of $T\cap B_{M-r}(u)$ such that
\[
B_Y\cap B\neq \0.
\]
This is well-defined because $u\in T\cap B_{M-r}(u)$ and $B_u\cap B\neq \0$.

For each $x\in T\cap B_{M-r}(u)$, the event $Y=x$ is equivalent to:
\begin{itemize}
    \item $t(x)\ge \dist(u,x)-r$;
    \item $t(y)<\dist(u,y)-r$ for all $y\prec x$ in $T\cap B_{M-r}(u)$.
\end{itemize}
If $A_{\mathrm{near}}$ occurs and $Y=x$, then $B_Y$ cuts $B$, so in particular
\[
t(x)<\dist(u,x)+r.
\]
Together with the condition $t(x)\ge \dist(u,x)-r$, this forces the integer-valued variable $t(x)$ to lie in an interval of length at most $2r+1$. Therefore,
\[
\mb P[A_{\mathrm{near}}\,|\,Y=x]
\le
1-(1-p)^{2r+1}
\le
(2r+1)p
\le
3rp.
\]
Summing over all possible values of $Y$, we obtain
\[
\mb P[A_{\mathrm{near}}]
=
\sum_{x\in T\cap B_{M-r}(u)}
\mb P[Y=x]\,
\mb P[A_{\mathrm{near}}\,|\,Y=x]
\le
3rp.
\]

Finally,
\[
\mb P[B_r(u)\text{ is cut}]
\le
\mb P[A_{\mathrm{far}}]+\mb P[A_{\mathrm{near}}]
\le
15p+3rp
\le
18rp
\le
20rp.\qedhere
\]
\end{proof}

We can now phrase the existence of padded decompositions as a CSP to which the Lov\'asz Local Lemma applies. Fix an integer $M>1$, a real number $r>0$, and $m\in\N^+$. Define a CSP
\[
\mc A_m:T\to^?\{1,\ldots,M\}^m
\]
as follows. For each $u\in T$, let $A_{u,m}$ be the constraint with domain
\[
B_{M+r}(u)\cap T
\]
such that a function
\[
\mathbf t=(t_1,\ldots,t_m):T\to \{1,\ldots,M\}^m
\]
satisfies $A_{u,m}$ if and only if $B_r(u)$ is \emph{not} cut in at least one of the families
\[
\mc P_{t_1},\ldots,\mc P_{t_m}.
\]
Then
\[
\mc A_m:=\{A_{u,m}:u\in T\}
\]
is a CSP. We equip $\{1,\ldots,M\}^m$ with the product measure of $m$ copies of $\mr{Tgeo}(p,M)$, so that the coordinates $t_1,\ldots,t_m$ are i.i.d. with distribution $\mr{Tgeo}(p,M)$.

\begin{lemma}\label{lem:CSP}
Under the assumptions of Lemma~\ref{lem:cut}, the CSP defined above satisfies
\[
\p(\mc A_m)\le (20rp)^m
\]
and
\[
\d(\mc A_m)\le (2M+2r)^b-1.
\]
\end{lemma}

\begin{proof}
The bound for $\p(\mc A_m)$ follows directly from Lemma~\ref{lem:cut}. For the dependency degree, note that if
\[
\bigl(B_{M+r}(u)\cap T\bigr)\cap \bigl(B_{M+r}(v)\cap T\bigr)\neq \0,
\]
then necessarily
\[
\dist(u,v)\le 2M+2r.
\]
Therefore,
\[
\d(\mc A_m)\le \gamma(2M+2r)-1\le (2M+2r)^b-1.
\]
The term $-1$ excludes the case $u=v$.
\end{proof}

\subsection{Asymptotic dimension of volume noncollapsed metric measure spaces of polynomial growth}

\begin{theorem}\label{thm:ExiPad}
Let $(X,\dist)$ be a complete separable metric space. Fix constants $\eps>0$ and $b\ge 0$. Set
\[
m:=\floor b+1,
\qquad
\alpha:=(1+\eps)\frac{m}{m-b}.
\]
Assume that
\[
r>
\max\left\{
9,\,
\bigl(8(b+1)(4b+9)\bigr)^{2/\alpha},\,
e^{1/(8\alpha(b+1))},\,
\left(\frac{8000\alpha(b+1)}{\eps}\right)^{2/\eps}
\right\}.
\]
If $\gamma(s)\le s^b$ for all $s\ge r$,
then $X$ admits an $(r,r^\alpha)$-padded decomposition with $m$ layers.
\end{theorem}

\begin{proof}
Let $T$ be a $(1,1)$-net of $X$. Choose
\[
p=\frac{8\alpha(b+1)\ln r}{r^\alpha},
\qquad
M=\left\lfloor 4(b+1)\frac{1}{p}\ln\frac{1}{p}\right\rfloor.
\]
We apply the randomized ball-carving construction above together with Lemmas~\ref{lem:cut} and \ref{lem:CSP}.

Note that if $y\ge 2$ and $z\ge y^2$, then $z/\ln z\ge y$. Applying this with
\[
y=8(b+1)(4b+9),\qquad z=r^\alpha,
\]
we obtain
\[
\frac{r^\alpha}{8\alpha(b+1)\ln r}
=
\frac{r^\alpha}{8(b+1)\ln(r^\alpha)}
\ge 4b+9.
\]
Hence
\[
p\le \frac{1}{4b+9},
\]
so the assumptions of Lemma~\ref{lem:cut} are satisfied.

Let $\mc A_m$ be the CSP defined above. We also observe that
\begin{equation}\label{eq:2M}
2M
\le
\frac{8(b+1)\ln(1/p)}{p}
=
\frac{r^\alpha\ln(1/p)}{\alpha\ln r}
\le
r^\alpha.
\end{equation}
Indeed, the last inequality follows from $1/p\le r^\alpha$, which is equivalent to
\[
8\alpha(b+1)\ln r\ge 1,
\]
and this holds because
\[
r>e^{1/(8\alpha(b+1))}.
\]

Since also $r\le r^\alpha$, Lemma~\ref{lem:CSP} and Theorem~\ref{thm: LLL} imply that $\mc A_m$ has a solution provided that
\[
e(20rp)^m(2M+2r)^b
<
e(20rp)^m(5r^\alpha)^b
<
1.
\]
The last inequality is equivalent to
\[
\frac{1}{p}
>
20\,e^{1/m}5^{b/m}\,r^{1+\alpha b/m}.
\]
Since $m>b$ and $m\ge 1$, we have
\[
e^{1/m}5^{b/m}\le 25,
\]
so it suffices to prove that
\[
\frac{r^\alpha}{8\alpha(b+1)\ln r}
>
500\,r^{1+\alpha b/m}.
\]
Because
\[
\alpha-\left(1+\frac{\alpha b}{m}\right)
=
\alpha\left(1-\frac{b}{m}\right)-1
=
(1+\eps)-1
=
\eps,
\]
it is enough to show that
\[
\frac{r^\eps}{\ln r}>4000\alpha(b+1).
\]
Using the elementary inequality
\[
\ln r\le \frac{2}{\eps}r^{\eps/2},
\]
it suffices to require
\[
r^{\eps/2}>\frac{8000\alpha(b+1)}{\eps},
\]
which is exactly ensured by the assumed lower bound on $r$.

Therefore the Lov\'asz Local Lemma applies, and we obtain a solution
\[
\mathbf t=(t_1,\ldots,t_m):T\to \{1,\ldots,M\}^m.
\]
By the ball-carving scheme, each family $\mc P_{t_i}$ has the property that its traces on $T$ partition $T$. Since $\mathbf t$ is a solution to $\mc A_m$, for every $u\in T$ the ball $B_r(u)$ is contained in one cluster in at least one layer. Finally, every cluster in every $\mc P_{t_i}$ has diameter at most $2M\le r^\alpha$ by \eqref{eq:2M}. Hence
\[
(\mc P_{t_1},\ldots,\mc P_{t_m})
\]
is an $(r,r^\alpha)$-padded decomposition with $m$ layers.
\end{proof}

Now we are ready to bound the asymptotic dimension of volume noncollapsed metric measure spaces of polynomial growth:

\begin{corollary}[Theorem~\ref{thm:measure}]
Let $(X,\dist,\meas)$ be a metric measure space. If $(X,\dist,\meas)$ is proper, has polynomial volume growth, and is volume noncollapsed, then
\[
\asdim(X)\le \floor{\rho^V(X)}.
\]
\end{corollary}

\begin{proof}
By Lemma~\ref{lem:FiniteNet}\eqref{item:LemF2}, we have
\[
\rho(X)\le \rho^V(X)<\infty.
\]
Set
\[
\rho:=\rho(X),
\qquad
m:=\floor{\rho}+1.
\]
Choose numbers $b_0,b$ such that
\[
\rho<b_0<b<m.
\]
By the definition of $\rho(X)$, there exists $r_0$ such that for all $s\ge r_0$,
\[
\gamma(s)\le (s+1)^{b_0}.
\]
After enlarging $r_0$ if necessary, we may also assume that
\[
(s+1)^{b_0}\le s^b
\qquad\text{for all } s\ge r_0.
\]
Fix $\eps>0$ and let
\[
\alpha:=(1+\eps)\frac{m}{m-b}.
\]
Then Theorem~\ref{thm:ExiPad} yields, for all sufficiently large $r$, an $(r,r^\alpha)$-padded decomposition with $m$ layers. By Corollary~\ref{cor:PadtoBas}, we infer that
\[
\asdim(X)\le m-1=\floor{\rho(X)}.
\]
Using again Lemma~\ref{lem:FiniteNet}\eqref{item:LemF2}, we obtain
\[
\rho(X)\le \rho^V(X),
\]
and therefore
\[
\asdim(X)\le \floor{\rho^V(X)}.
\]
\end{proof}

\section{Nilmanifolds and equality in Theorem~\ref{thm:measure}}\label{sec:rigidity}

In this section we review some standard facts about nilpotent groups and nilmanifolds in order to analyze equality in Theorem~\ref{thm:measure}. More precisely, we consider the universal cover of a closed nilmanifold. If the universal cover of a closed manifold has polynomial volume growth, then so does the word-metric growth of the fundamental group of the base manifold. By Gromov's theorem \cite{GrmovPoly}, this group is virtually nilpotent, so nilmanifolds arise naturally in this setting.

In this class of examples we may replace asymptotic dimension by Assouad--Nagata dimension, since they coincide for finitely generated nilpotent groups. Along the way, we also see that universal covers of nilmanifolds provide natural examples in which asymptotic dimension and Assouad--Nagata dimension do not, in general, recover the polynomial growth rate.

Let $G$ be a finitely generated group, always equipped with a word metric. Define inductively
\[
G_0\defeq G,
\qquad
G_{i+1}\defeq [G,G_i],
\qquad i\in \N.
\]
Then $G_{i+1}\trianglelefteq G_i$ and $G_i/G_{i+1}$ is Abelian. We call $G$ \emphd{nilpotent} if there exists $s\in\N^+$ such that $G_s=\{\mr{id}\}$. The smallest such $s$ is the \emphd{step} of $G$.

We define the \emphd{rank} of a finitely generated nilpotent group $G$ by
\[
\mr{rank}(G)\defeq \sum_{i=0}^{s-1}\mr{rank}(G_i/G_{i+1}),
\]
where $\mr{rank}(G_i/G_{i+1})$ denotes the rank of the finitely generated Abelian group $G_i/G_{i+1}$. This is the Hirsch length of $G$. Recall that rank is additive in short exact sequences of finitely generated nilpotent groups.

We also define the \emphd{homogeneous dimension} of $G$ by
\[
\dim_H(G)\defeq \sum_{i=0}^{s-1}(i+1)\mr{rank}(G_i/G_{i+1}).
\]
Then $\dim_H(G)\ge \mr{rank}(G)$, with equality if and only if $G$ is Abelian. Two standard facts are relevant here:
\begin{itemize}
    \item by Bass' theorem \cite{Bass}*{Theorem~2}, the metric growth rate of $G$ is
    \[
    \rho(G)=\dim_H(G);
    \]
    \item by \cite{HigesPeng}*{Corollary~5.10} and \cite{bell2008asymptotic}*{Corollary~68},
    \[
    \andim(G)=\mr{rank}(G)=\asdim(G).
    \]
\end{itemize}
Thus, for a finitely generated non-Abelian nilpotent group, the polynomial growth rate is strictly larger than both its asymptotic dimension and its Assouad--Nagata dimension.

\begin{example}
The $3$-dimensional integer Heisenberg group $H^3(\Z)$ satisfies
\[
\dim_H(H^3(\Z))=\rho(H^3(\Z))=4,
\qquad
\asdim(H^3(\Z))=\andim(H^3(\Z))=\mr{rank}(H^3(\Z))=3.
\]
\end{example}

We now pass from nilpotent groups to nilmanifolds.

\begin{definition}
A closed manifold $M$ is a \emphd{nilmanifold} if it is a quotient
\[
M=L/\Gamma,
\]
where $L$ is a simply connected nilpotent Lie group and $\Gamma\le L$ is a discrete cocompact subgroup.
\end{definition}

Then $\pi_1(M)=\Gamma$ and $L$ is the universal cover $\widetilde M$ of $M$. If $L$ is equipped with a left-invariant Riemannian metric, then the action of $\Gamma$ on $L$ is by isometries, and this metric descends to a Riemannian metric on $M$. Recall also that nilmanifolds are almost flat, and conversely every almost flat manifold is a infranilmanifold, in particular, it is finitely covered by a nilmanifold \cites{AFGromov,AFRuh}.

The \v{S}varc--Milnor lemma reveals that a nilmanifold is quasi-isometric to a nilpotent group, its fundamental group.

\begin{lemma}[\v{S}varc--Milnor lemma \cite{bell2008asymptotic}*{Theorem~51}]\label{lem:MS}
Let $(M,g)$ be a closed Riemannian manifold. Then $\pi_1(M)$ with word metric is quasi-isometric to $(\widetilde M,\widetilde g)$, and
\[
\rho^V(\widetilde M)=\rho(\pi_1(M)).
\]
In particular, for any point $x\in M$ and any lift $\widetilde x$, the orbit map
\[
f:\pi_1(M)\to \widetilde M,
\qquad
f(h)=h\widetilde x,
\]
is a quasi-isometry.
\end{lemma}

Since the universal covering map is a local isometry, $(\widetilde M,\widetilde g)$ is automatically volume noncollapsed. This allows us to characterize the equality case in Theorem~\ref{thm:measure} for nilmanifolds.

\begin{proposition}\label{prop:Rigidity}
Let $M=L/\Gamma$ be a nilmanifold. The following are equivalent:
\begin{enumerate}
    \item\label{item:asim=vol} for some (equivalently, every) Riemannian metric on $M$,
    \[
    \rho^V(\widetilde M)=\asdim(\widetilde M)=\andim(\widetilde M);
    \]
    \item\label{item:dh=rank}
    \[
    \dim_H(\pi_1(M))=\mr{rank}(\pi_1(M));
    \]
    \item\label{item:pi1Abelian} $\pi_1(M)$ is Abelian;
    \item\label{item:Diffeo} $M$ is diffeomorphic to a torus.
\end{enumerate}
\end{proposition}

\begin{proof}
The equivalence \eqref{item:dh=rank} $\Leftrightarrow$ \eqref{item:pi1Abelian} is immediate from the definition of the homogeneous dimension:
\[
\dim_H(G)=\sum_{i=0}^{s-1}(i+1)\mr{rank}(G_i/G_{i+1}),
\qquad
\mr{rank}(G)=\sum_{i=0}^{s-1}\mr{rank}(G_i/G_{i+1}),
\]
so equality holds if and only if $s=1$, i.e. $G$ is Abelian.

We next prove \eqref{item:asim=vol} $\Leftrightarrow$ \eqref{item:dh=rank}. By Lemma~\ref{lem:MS}, the universal cover $\widetilde M$ is quasi-isometric to $\pi_1(M)$, hence
\[
\asdim(\widetilde M)=\asdim(\pi_1(M)).
\]
Since $\pi_1(M)$ is finitely generated nilpotent, we have by \cite{bell2008asymptotic}*{Theorem~71} that
\[
\asdim(\pi_1(M))=\mr{rank}(\pi_1(M)),
\]
and likewise by \cite{HigesPeng}*{Corollary~5.10} that
\[
\andim(\pi_1(M))=\mr{rank}(\pi_1(M)).
\]
On the other hand, Lemma~\ref{lem:MS} also gives
\[
\rho^V(\widetilde M)=\rho(\pi_1(M)),
\]
and \cite{Bass}*{Theorem~2} implies
\[
\rho(\pi_1(M))=\dim_H(\pi_1(M)).
\]
Therefore,
\[
\rho^V(\widetilde M)=\asdim(\widetilde M)=\andim(\widetilde M)
\]
if and only if
\[
\dim_H(\pi_1(M))=\mr{rank}(\pi_1(M)).
\]

Finally, we prove \eqref{item:pi1Abelian} $\Leftrightarrow$ \eqref{item:Diffeo}. The implication \eqref{item:Diffeo} $\Rightarrow$ \eqref{item:pi1Abelian} is obvious. Conversely, assume that $\Gamma=\pi_1(M)$ is Abelian. Since $L$ is connected, simply connected and nilpotent, $\exp: \mathfrak l\to L$ is a diffeomorphism with inverse $\log: L\to \mathfrak l$, where $\mathfrak l$ is the Lie algebra of $L$. Since $\Gamma$ is a cocompact lattice subgroup, $\log(\Gamma)$ spans $\mathfrak l$. Because $\Gamma$ is Abelian, $\log(\Gamma)$ is also Abelian, hence $\mathfrak l$ itself is Abelian. Therefore
\[
\Gamma\cong \Z^n
\qquad\text{and}\qquad
L\cong \R^n.
\]
It follows that
\[
M=L/\Gamma \cong \R^n/\Z^n=\mathbb T^n,
\]
so $M$ is diffeomorphic to a torus.
\end{proof}

\appendix

\appendix
\section{A simplicial argument \`a la Gromov and largeness}\label{sec:Gromov}

In this appendix we record a complementary simplicial argument in the spirit of Gromov. It yields the sharp upper bound of the Assouad--Nagata dimension and is closely related to the largeness discussion in Proposition~\ref{prop:Euc} for asymptotic Assouad--Nagata dimension. 

Let $(X,\dist)$ be a metric space. A simplicial complex $P$ is called \emphd{uniform} if it carries the metric induced by the restriction of the Euclidean metric on $\ell^2(P^0)$, where $P^0$ denotes the vertex set of $P$. For $D>0$, a map $f:X\to P$ is \emphd{$D$-cobounded} if
\[
\diam f^{-1}(\sigma)\le D
\]
for every simplex $\sigma\subset P$.

\begin{proposition}[Dranishnikov--Smith \cite{andim}*{Proposition~1.6, 1.7}]\label{prop:asanequi}
Let $(X,\dist)$ be a metric space. Then:
\begin{enumerate}
    \item $\asdim(X)\le n$ if for every $\eps>0$ there exists $D(\eps)>0$ and a uniform simplicial complex $P$ of dimension $n$ together with an $\eps$-Lipschitz, $D(\eps)$-cobounded map $\phi:X\to P$;
    \item $\asandim(X)\le n$ if there exist $\bar\eps>0$ and $D>0$ such that for every $0<\eps<\bar\eps$ there exists a uniform simplicial complex $P$ of dimension $n$ and an $\eps$-Lipschitz, $D/\eps$-cobounded map $\phi:X\to P$;
    \item $\andim(X)\le n$ if there exists $D>0$ such that for every $\eps>0$ there exists a uniform simplicial complex $P$ of dimension $n$ and an $\eps$-Lipschitz, $D/\eps$-cobounded map $\phi:X\to P$.
\end{enumerate}
\end{proposition}

\begin{remark}
In \cite{andim}, the statement corresponding to item~(2) is formulated for the \emph{asymptotic} Assouad--Nagata dimension, which accounts for the upper bound $\bar\eps$. Item~(3) is the uniform version relevant to the Assouad--Nagata dimension.
\end{remark}

For reader's convenience we review the history. Gromov \cite{VolBddCohom} first proved the sharp macroscopic dimension bound $\dim_{\mathrm{ms}}\le n\in \N$ for $n$-manifolds under nonnegative sectional curvature assumption. This argument was observed to hold for manifolds of nonnegative Ricci curvature by \cites{CaiLargemanifold,ShenLarge}. The argument only needs (global) volume doubling property to control intersection of balls, which is provided by the Bishop--Gromov volume comparison theorem, a consequence of nonnegative Ricci curvature. The sharp bound $\dim_{\mathrm{ms}}\le n$ comes from, again, the fact that $\cd\le 2^n$ and $\dim_{\mathrm{ms}}\le \floor{\log_2 \cd}$.

The macroscopic dimension of a metric space $X$ is defined as the smallest dimension of a simplicial complex $P$ such that there exists a continuous and cobounded map from $X$ to $P$. Proposition \ref{prop:asanequi} pointed out the relation between the macroscopic dimension, asymptotic dimension and asymptotic Assouad--Nagata dimension:
\[
\dim_{\mathrm{ms}}\le \asdim\le \asandim\le \andim.
\]
Indeed, in the definition of the macroscopic dimension, the metric on the simplicial complex is not required to be uniform and the continuous map from $X$ to $P$ is not required to be Lipschitz. However when Gromov proved the macroscopic dimension bound, the simplicial complex he constructed is in fact uniform, and the continuous function from $X$ to $P$ is in fact Lipschitz. Since his goal was the macroscopic dimension, he did not emphasize these facts. With Proposition \ref{prop:asanequi}, Gromov's argument provides the stronger sharp Assouad--Nagata dimension upper bound $\andim\le \floor{\log_2 \cd}$ hence solves the conjecture of Papasoglu \cite{papasoglu2021polynomial}.    

We point out that Gromov's argument also recovers the Theorem of Rajala--Le Donne \cite{RajalaLeDonne}. That is, the volume doubling condition can be replaced by the metric doubling condition, to provide a similar sharp Assouad--Nagata dimension in terms of the metric doubling constant, as the two doubling conditions control the intersection of balls in the same way.  

\begin{theorem}[Gromov \cite{VolBddCohom}*{Section~3.4}]\label{thm:cobddLip}
Let $(X,\dist)$ be a doubling metric space with doubling constant $N$. Then there exists $D=D(N)$ such that for every $\eps>0$ there exist a uniform simplicial complex $P$ of dimension $\floor{\log_2 N}$ and an $\eps$-Lipschitz, $D/\eps$-cobounded map
\[
F:X\to P.
\]
In particular $\andim(X)\le \floor{\log_2 N}$.
\end{theorem}

\begin{proof}
Fix $\eps>0$ and set $r=\eps^{-1}$. Take a $(2r/3,2r/3)$-net $T=\{x_i\}_{i=1}^{\infty}$. Then the family $\{B_{2r/3}(x_i)\}_{i=1}^{\infty}$ covers $X$, while the balls $\{B_{r/3}(x_i)\}_{i=1}^{\infty}$ are pairwise disjoint. Define
\[
\phi_i(x)\defeq \max\{0,1-3r^{-1}\dist(x,B_{2r/3}(x_i))\}.
\]
Each $\phi_i$ is a nonnegative $9r^{-1}$-Lipschitz function. Now define
\[
F_0:X\to \ell^2(\N),
\qquad
F_0(x)=\left(\frac{\phi_j(x)}{\sum_i \phi_i(x)}\right)_{j=1}^{\infty}.
\]
Its image lies in the unit simplex
\[
\Delta\defeq \Bigl\{(y_i)_{i=1}^{\infty}\in \ell^2(\N): \sum_i y_i=1\Bigr\}.
\]
By Lemma~\ref{lem:DMgrowth}, there exists $d=d(N)$ such that every $x\in X$ belongs to at most $d$ of the balls $B_{2r/3}(x_i)$.  By identifying $\N$ with the net $T$, $F_0$ can be viewed as a map into the nerve of the covering $\{B_{2r/3}(x_i)\}_{i=1}^\infty$, which is realized as a subcomplex $P^{(d)}$ of dimension at most $d$ in $\Delta$. A direct computation shows that $F_0$ is $Cr^{-1}$-Lipschitz for some $C=C(N)$.

We now perform Gromov's dimension-reduction step. Inductively define maps
\[
F_j:X\to P^{(d-j)},\qquad j\ge 0,
\]
with the following properties:
\begin{enumerate}
    \item\label{item:lip} there exists $C_j=C_j(N)$ such that $F_j$ is $C_jr^{-1}$-Lipschitz;
    \item\label{item:intersection} for every vertex $v$ of $P^{(d-j)}$, the pullback $F_j^{-1}(\st(v))$ is contained in a union of finitely many balls from $\{B_{2r/3}(x_i)\}$ intersecting a fixed ball of this family. In particular,
    \[
    \sup_{\sigma\subset P^{(d-j)}}\diam(F_j^{-1}(\sigma))\le 4r.
    \]
\end{enumerate}
These properties hold for $j=0$ by construction.

Assume $\dim P^{(d-j)}>\log_2 N$. For each top-dimensional simplex $\sigma\subset P^{(d-j)}$, we have the following claim:

\begin{claim}
    There exists a point in $\Int(\sigma)\setminus F_j(X)$ at uniformly positive distance from $F_j(X)\cap \sigma$.
\end{claim} 

\begin{proof}[Proof of claim]
    Let
\[
\delta_\sigma\defeq \sup_{y\in \sigma}\dist(y,F_j(X)\cap \sigma)\ge 0.
\]
Fix any positive number $\eps_0\ge \delta_\sigma$. Take a finite $(\eps_0,\eps_0)$-net $\{y_i\}_{i=1}^k$ in $F_j(X)\cap \sigma$. Then $\{B_{2\eps_0}(y_i)\}_{i=0}^k$ covers $\sigma$. Since $\sigma\in P^{(d-j)}$ carries the Euclidean metric, there exists $c\defeq c(N)>0$ such that 
\[
\vol(\sigma)\le k c \eps_0^{\dim\sigma}\Rightarrow\ k\ge c \eps_0^{-\dim\sigma}.
\] 
By item \eqref{item:lip}, every $F_j^{-1}(B_{\eps_0}(y_i))$ contains a ball of radius $C^{-1}r\eps_0$, and by construction, the sets in the family $\{F_j^{-1}(B_{\eps_0}(y_i))\}_{i=0}^k$ are pairwise disjoint. Thus, we have a $Cr\eps_0$-separated set of $k$ elements. Meanwhile, by item \eqref{item:intersection}, $\cup_{i=1}^k F_j^{-1}(B_{\eps_0}(y_i))$ are contained in the union of at most $d$ balls of radius $2r/3$ that intersects a fixed ball $B_{2r/3}$. Therefore, this $Cr\eps_0$-separated set is contained in the ball $B_{2r}$. By Lemma \ref{lem:DMgrowth}, we have 
\[
k\le N^2 \left(\frac{2r}{Cr\eps_0}\right)^{\log_2 N} =C\eps_0^{-\log_2 N}.
\]
Combining two bounds for $k$ gives
\[
c \eps_0^{-\dim \sigma}\le k\le C\eps_0^{-\log_2 N}.
\]
Since $\dim\sigma=\dim P^{(d-j)}>\log_2 N$, we get a uniformly positive lower bound of $\eps_0$, in turn, as $\eps_0$ is arbitrarily close to $\delta_\sigma$, we have
\[
\delta_\sigma\ge c(N)>0,
\]
independent of $\sigma$. 
%Choose a finite $(\delta_\sigma,\delta_\sigma)$-net in $F_j(X)\cap \sigma$. Volume comparison inside the Euclidean simplex $\sigma$ yields a lower bound of order $\delta_\sigma^{-\dim \sigma}$ on the size of such a net. On the other hand, by the Lipschitz property, disjoint preimages of the corresponding balls yield a separated subset of a metric ball in $X$, and Lemma~\ref{lem:DMgrowth} bounds its size by a multiple of $\delta_\sigma^{-\log_2 N}$. Since $\dim\sigma>\log_2 N$, this forces a positive lower bound
\end{proof}

Choose the point in the claim in each top-dimensional simplex $\sigma$ and project linearly from that point to $\partial\sigma$, the boundary of the simplex. These simplexwise projections glue to a Lipschitz map
\[
\pi_j:P^{(d-j)}\to P^{(d-j-1)}
\]
with Lipschitz constant depending only on $N$. Set
\[
F_{j+1}\defeq \pi_j\circ F_j.
\]
Then $F_{j+1}$ still satisfies the two properties above. Iterating this procedure until the target has dimension $\floor{\log_2 N}$ yields a map
\[
F:X\to P
\]
into a uniform simplicial complex $P$ of dimension $\floor{\log_2 N}$ such that $F$ is $Cr^{-1}$-Lipschitz and $4r$-cobounded for some $C=C(N)$. Since $r=\eps^{-1}$, this proves the theorem.
\end{proof}

%\begin{remark}
%Combining Theorem~\ref{thm:cobddLip} with Proposition~\ref{prop:asanequi} recovers the sharp doubling bound. The same simplicial construction may also be run in the volume-doubling setting, leading to Theorem~\ref{thm:VD}.
%\end{remark}

Finally, we observe that with additional volume noncollapsed assumption, one can further use Gromov's argument to study the case of equality for the asymptotic Assouad--Nagata dimension upper bound.

\begin{proposition}[Proposition~\ref{prop:Euc}]\label{prop:large}
Let $(M,g)$ be a Riemannian $n$-manifold with $\Ric_g\ge 0$ and volume noncollapsed, i.e.
\[
v\defeq \inf_{p\in M}\vol_g(B_1(p))>0.
\]
Then $\asandim(M)=n$ (or $\asdim(M)=n$) if and only if $M$ is large in the sense of Gromov, namely
\[
\sup_{x\in M}\vol_g(B_r(x))=\omega_n r^n
\qquad\text{for every }r\ge 0.
\]
\end{proposition}

We start with a lemma that characterizes the largeness of a Riemannian manifold in terms of its asymptotic volume growth.

\begin{lemma}\label{lem:largeeq}
Let $(M,g)$ be a complete Riemannian manifold with $\Ric_g\ge 0$. Then the following are equivalent:
\begin{enumerate}
    \item\label{item:large} $M$ is large, i.e.
    \[
    \sup_{p\in M}\vol_g(B_r(p))=\omega_n r^n
    \qquad\text{for every }r\ge 0;
    \]
    \item\label{item:growth} The volume growth of balls in $M$ satisfies
    \[
    \limsup_{r\to \infty} \frac{\sup_{p\in M}\vol_g(B_r(p))}{r^n}>0.
    \]
\end{enumerate}
\end{lemma}

\begin{proof}
The implication \eqref{item:large}$\Rightarrow$\eqref{item:growth} is immediate. Assume \eqref{item:growth}. Then for some $c>0$ there exist $r_i\to\infty$ and $p_i\in M$ such that
\[
\vol_g(B_{r_i}(p_i))\ge c r_i^n.
\]
By Bishop--Gromov, this implies
\[
\vol_g(B_r(p_i))\ge c r^n
\qquad\text{for all }0\le r\le r_i.
\]
Thus $(M,p_i,\vol_g)$ is a volume-noncollapsing sequence, and after passing to a subsequence it converges in the pointed Gromov--Hausdorff sense to a Ricci limit space $(M_\infty,p_\infty,\haus^n)$. Volume convergence \cites{Cheeger-Colding97I} implies
\[
\haus^n(B_r(p_\infty))\ge c r^n
\qquad\text{for every }r\ge 0.
\]
The argument used in \cite{ShenLarge}*{Theorem~1.3}, together with the splitting theorem for Ricci limit spaces \cite{Cheeger-Colding97I}, yields a sequence of points $q_i\in M_\infty$ such that $(M_\infty,q_i)$ converges to $(\R^n,0)$. Pulling these points back and applying a diagonal argument, one obtains points $q_i'\in M$ with
\[
(M,q_i')\to (\R^n,0).
\]
A second application of volume convergence shows that for every fixed $r\ge 0$,
\[
\lim_{i\to\infty}\vol_g(B_r(q_i'))=\omega_n r^n.
\]
Hence $M$ is large.
\end{proof}

\begin{proof}[Proof of Proposition~\ref{prop:large}]
Assume first that $M$ is large. Then, by \cite[Theorem~1]{CaiLargemanifold} together with the discussion below \cite[Theorem~1.3]{ShenLarge}, there is no continuous cobounded map from $M$ to an $(n-1)$-dimensional simplicial complex. In particular,
\[
\asandim(M)\ge \dim_{\mathrm{ms}}(M)=n.
\]
On the other hand, $\Ric_g\ge 0$ implies that $(M,g)$ is volume doubling with doubling constant at most $2^n$. Hence Theorem~\ref{thm:VD} gives
\[
\asandim(M)\le \andim(M)\le n.
\]
Therefore $\asandim(M)=n$.

Conversely, suppose that $\asandim(M)=n$ and that $M$ is not large. Let
\[
V(r)\defeq \sup_{x\in M}\vol_g(B_r(x)).
\]
By Lemma~\ref{lem:largeeq} the function $\frac{V(r)}{r^n}$ tends to $0$ as $r\to\infty$.

Fix $\eps>0$ and set $r:=\eps^{-1}$. Applying the simplicial construction from the proof of Theorem~\ref{thm:cobddLip} in the volume-doubling setting, we obtain an $\eps$-Lipschitz map
\[
F_\eps:M\to P^n
\]
to an $n$-dimensional uniform simplicial complex with the following additional property: there exists a constant $C_n>0$ such that for every vertex $v$ of $P^n$, the pullback
$F_\eps^{-1}(\st(v))$
is contained in a union of at most $C_n$ balls of radius $C_n r$ intersecting a fixed ball of the same radius. In particular, the preimage of every simplex has diameter at most $C_n r$.

For each top-dimensional simplex $\sigma\subset P^n$, define
\[
\delta_\sigma(\eps)\defeq \sup_{y\in \sigma} \dist\bigl(y,F_\eps(M)\cap \sigma\bigr)>0,
\qquad
\delta(\eps)\defeq \inf_\sigma \delta_\sigma(\eps)>0.
\]
We claim that
\[
\frac{\delta(\eps)}{\eps}\longrightarrow \infty
\qquad\text{as }\eps\to 0.
\]

Assume not. Then there exist a constant $C>0$ and a sequence $\eps_j\to 0$ such that
\[
\delta(\eps_j)\le C\eps_j.
\]
For each $j$, choose a top-dimensional simplex $\sigma_j$ such that
\[
\delta_{\sigma_j}(\eps_j)\le 2\delta(\eps_j).
\]
Let
\[
\{y_1,\dots,y_{k_j}\}\subset F_{\eps_j}(M)\cap \sigma_j
\]
be a maximal $4\delta(\eps_j)$-separated set. Since every point of $\sigma_j$ lies within distance $2\delta(\eps_j)$ of $F_{\eps_j}(M)\cap \sigma_j$, the $10\delta(\eps_j)$-balls centered at the $y_i$ cover $\sigma_j$. Hence
\[
k_j\ge c_n\,\delta(\eps_j)^{-n}
\]
for a constant $c_n>0$ depending only on $n$.

Choose $x_i\in M$ such that $F_{\eps_j}(x_i)=y_i$. Since $F_{\eps_j}$ is $\eps_j$-Lipschitz and the $y_i$ are $4\delta(\eps_j)$-separated, the balls
\[
B_{\delta(\eps_j)/(4\eps_j)}(x_i)
\]
are pairwise disjoint. On the other hand, by the star-control property above, all these balls are contained in a ball of radius $C_n\eps_j^{-1}$. It follows that 
\[
k_j \min_{1\le i\le k_j} \vol_g\left(B_{\delta(\eps_j)/(4\eps_j)}(x_i)\right)\le \sum_{i=1}^{k_j}\vol_g\left(B_{\delta(\eps_j)/(4\eps_j)}(x_i)\right)\le V(C_n\eps_j^{-1}).
\]
Combing two bounds for $k_j$ gives that 
\[
 \frac{\inf_{z\in M} \vol_g\left(B_{\delta(\eps_j)/(4\eps_j)}(z)\right)}{ c_n\,\delta(\eps_j)^{n}/(4\eps_j)^n}\cdot\frac{1}{(4\eps_j)^n}=\frac{\inf_{z\in M} \vol_g\left(B_{\delta(\eps_j)/(4\eps_j)}(z)\right)}{ c_n\,\delta(\eps_j)^{n}}\le  V(C_n\eps_j^{-1}) 
\]
We derive a contradiction in two cases. Let $z_j\in M$ be a point that almost realizes the infimum of the volume, that is 
$$\vol_g\left(B_{\delta(\eps_j)/(4\eps_j)}(z_j)\right)\le 2\inf_{z\in M} \vol_g\left(B_{\delta(\eps_j)/(4\eps_j)}(z)\right).$$ 

If $\delta(\eps_j)/(4\eps_j)\le 1$, then using volume noncollapse and Bishop--Gromov, we obtain
\[
\frac{v}{2c_n}\le \vol_g(B_{1}(z_j))/2c_n\le \frac{\vol_g\left(B_{\delta(\eps_j)/(4\eps_j)}(z_j)\right)}{ 2c_n\,\delta(\eps_j)^{n}/(4\eps_j)^n}\le \frac{V(C_n\eps_j^{-1})}{(4\eps_j)^{-n}}\to 0
\]
as $j\to\infty$, a contradiction. 

If $\delta(\eps_j)/(4\eps_j)\ge 1$, then $4\eps_j\le \delta(\eps_j)\le C\eps_j$. Passing to a possible subsequence as $j\to\infty$, we assume that $ \delta(\eps_j)/(4\eps_j)\to s\in  [1,C]$, and that $z_j$ either converges to a point $z\in M$ or a point $z\in M_\infty$, the Ricci limit space of the sequence $(M,z_j)$. In either case we have by volume convergence that 
\[
0=\limsup_{r\to\infty }\frac {V(r)}{r^n}\ge \liminf_{j\to \infty}\frac{\vol_g\left(B_{\delta(\eps_j)/(4\eps_j)}(z_j)\right)}{ 2c_n\,\delta(\eps_j)^{n}/(4\eps_j)^n}=2c_n s^{-n}\haus^n(B_s(z))>0,
\]
a contradiction.

Therefore, we have proved the claim that
\[
\frac{\delta(\eps)}{\eps}\to\infty.
\]

Now, for each top-dimensional simplex $\sigma\subset P^n$, choose a point
\[
a_\sigma\in \Int(\sigma)\setminus F_\eps(M)
\]
such that
\[
d\bigl(a_\sigma,F_\eps(M)\cap \sigma\bigr)\ge \delta(\eps)/2.
\]
Let
\[
\pi_\sigma:\sigma\to \partial \sigma
\]
be the radial projection from $a_\sigma$. Since $P^n$ is uniform, there exists a constant $C_n'>0$ such that each $\pi_\sigma$ is $C_n'/\delta(\eps)$-Lipschitz. Gluing these maps simplexwise yields a map
\[
\pi:P^n\to P^{n-1}
\]
with the same Lipschitz bound. Moreover, for every vertex $v$ of $P^n$ we have
\[
\pi^{-1}(\st(v))\subseteq \st(v),
\]
so the star-control property for $F_\eps$ persists for $\pi\circ F_\eps$. Consequently,
\[
\pi\circ F_\eps:M\to P^{n-1}
\]
is $C_n'\eps/\delta(\eps)$-Lipschitz and $C_n''/\eps$-cobounded for some constant $C_n''$ independent of $\eps$.

Since $\eps/\delta(\eps)\to 0$, Proposition~\ref{prop:asanequi} implies that
\[
\asandim(M)\le n-1,
\]
contradicting the assumption $\asandim(M)=n$. Hence $M$ must be large.
\end{proof}

\begin{remark}
Under the same assumptions, the proof above also shows that $\asdim(M)=n$ if and only if $\asandim(M)=n$.
\end{remark}

\begin{corollary}[Corollary~\ref{cor:psc}]\label{cor:scalar}
Let $(M,g)$ be an $n$-dimensional complete noncompact manifold with $\Ric_g\ge 0$, $\Sc_g\ge 2$, and
\[
v\defeq \inf_{x\in M}\vol_g(B_1(x))>0.
\]
Then
\[
\asandim(M)\le n-1.
\]
\end{corollary}

\begin{proof}
Under our assumptions, it is proved in \cite[Theorem~1.1]{wxzzScalar} that $(M,g)$ is not large. The result therefore follows from Proposition~\ref{prop:large}.
\end{proof}

\bibliographystyle{alphaurl}
\bibliography{ref}
\end{document}